\documentclass[11pt,twoside,a4paper]{article}

\author{Shai Sarussi}
 \title{Alexandroff Topology of Algebras over an Integral Domain}
\date{}

\usepackage{amsmath,amsthm,amssymb}
\usepackage{verbatim}
\usepackage{amssymb}
\usepackage{graphicx}

\begin{document}

\newtheorem{thm}{Theorem}[section]
\newtheorem{cor}[thm]{Corollary}
\newtheorem{lem}[thm]{Lemma}
\newtheorem{prop}[thm]{Proposition}
\newtheorem{ax}{Axiom}

\theoremstyle{definition}
\newtheorem{defn}[thm]{Definition}

\theoremstyle{remark}
\newtheorem{rem}[thm]{Remark}
\newtheorem{ex}[thm]{Example}
\newtheorem*{notation}{Notation}
\newtheorem{conjecture}[thm]{Conjecture}

\newcommand{\qv}{{quasi-valuation\ }}
\newcommand{\s}{{$\mathbb S$\ }}

\def\twoheaddownarrow{\rlap{$\downarrow$}\raise-.4ex\hbox{$\downarrow$}}

\def\twoheaduparrow{\rlap{$\uparrow$}\raise-.4ex\hbox{$\uparrow$}}



\maketitle

\begin{abstract} {Let $S$ be an integral domain with field of fractions $F$ and let $A$ be an $F$-algebra.
An $S$-subalgebra $R$ of $A$ is called $S$-nice if $R$ is lying over $S$ and the localization of $R$ with respect to $S \setminus \{ 0 \}$ is $A$. Let $\mathbb S$ be the set of all $S$-nice subalgebras of $A$.
We define a notion of open sets on $\mathbb S$ which makes this set a $T_0$ Alexandroff space.
This enables us to study the algebraic structure of
$\mathbb S$ from the point of view of topology.
We prove that an irreducible subset of $\mathbb S$ has a supremum with respect to the specialization order. We present
equivalent conditions for an open set of $\mathbb S$ to be irreducible,
and characterize the irreducible components of $\mathbb S$.





}\end{abstract}

\section{Introduction and some preliminary results}

As the title suggests, in this paper we discuss algebras over integral domains from the point of view
of Alexandroff topology, which will shortly be defined.
In [Sa1] and [Sa2] we studied algebras over valuation domains, concentrating on quasi-valuations that extend a valuation on a field.
In [Sa3] we prove several existence theorems of integral domains that may be applied to the study of quasi-valuations.
More specifically, let $S$ be an integral domain which is not a field, let
$F$ be its field of fractions, and let $A \neq F$ be an $F$-algebra.
In [Sa3] we study $S$-subalgebras of $A$, that are
lying over $S$ and whose localizations with respect to $S \setminus \{ 0 \}$ is $A$.
We call them $S$-nice subalgebras of $A$.
Namely, an $S$-subalgebra $R$ of $A$ is called an $S$-nice subalgebra of $A$ if
$R \cap F=S$ and $FR=A$;
we shall use this notation throughout this paper.
We denote by $\mathbb S$ the set of all $S$-nice subalgebras of $A$.





We recall now some definitions and results from [Sa3].
The following concept is used quite frequently:
let $B$ be a basis of $A$ over $F$. We say that $B$ is $S$-stable
if there exists a basis $C$ of $A$ over $F$ such that for all $c \in C$ and $b \in B$,
one has $cb \in \sum_{y \in B}  Sy$.

We note in [Sa3, Remark 3.4] that
if a basis $B$ is closed under multiplication then $B$ is $S$-stable. Thus, for example,
every free (noncommutative) $F$-algebra with an arbitrary set of generators has an $S$-stable basis;
in particular, every polynomial algebra with an arbitrary set of indeterminates over $F$ has an $S$-stable basis.

We also show in [Sa3, Proposition 3.12] that if $A$ is finite dimensional over $F$, then every basis of $A$ over $F$ is $S$-stable.
The first existence theorem is as follows.

\begin{thm} \label{existence of S-nice} (cf. [Sa3, Theorem 3.14]) If there exists an $S$-stable basis of $A$ over $F$, then there exists an $S$-nice
subalgebra of $A$.

\end{thm}

In particular, if $A$ is finite dimensional over $F$ then there exists an $S$-nice
subalgebra of $A$.

The following result is a going-down lemma for $S$-nice subalgebras.

\begin{lem} \label{going down} (cf. [Sa3, Lemma 3.20]) Let $S_1 \subseteq S_2$ be integral domains with field of fractions $F$ such that $S_2 \neq F$.
Assume that there exists an $S_1$-stable basis of $A$ over $F$.
Let $R$ be an $S_2$-nice subalgebra of $A$. Then there exists an $S_1$-nice subalgebra of $A$, which is contained in $R$. 

\end{lem}

We conclude that a minimal $S$-nice subalgebra of $A $ does not exist. More precisely, we prove

\begin{prop} \label{not minimal} (cf. [Sa3, Proposition 3.21]) Assume that there exists an $S$-stable basis of $A$ over $F$.
Let $R$ be an $S$-nice subalgebra of $A$. Then there exists an infinite decreasing chain
of $S$-nice subalgebras of $A $ starting from $R$. In particular, a minimal $S$-nice subalgebra of $A $ does not exist.

\end{prop}


\begin{defn} Let $\mathcal C$ be a chain of prime ideals of $S$ and let $R$ be a faithful $S$-algebra.
Let $\mathcal D$ be a chain of prime ideals of $R$. We say that $\mathcal D$ covers $\mathcal C$ if for every $P \in \mathcal C$
there exists $Q \in \mathcal D$ lying over $P$; namely, $Q \cap S=P$.
\end{defn}

\begin{thm} \label{every chain is covered commutative case} (cf. [Sa3, Theorem 3.24]) Assume that there exists an $S$-stable basis of $A$ over $F$.
Let $\mathcal C$ be a chain of prime ideals of $S$.
If $A$ is commutative then there exists an $S$-nice subalgebra $R$ of $A$ such that there exists a chain of prime ideals
$\mathcal F$ of $R$ covering $\mathcal C$. In fact, there exists an infinite descending chain of such
$S$-nice subalgebras of $A$.
\end{thm}



Note that by Theorem \ref{existence of S-nice} and Proposition \ref{not minimal} if $A$ contains an $S$-stable basis then
$\mathbb S$ is not empty and is, in fact, infinite. In particular, by the above-mentioned note before Theorem \ref{existence of S-nice}, if $A$ is finite dimensional over $F$ then $\mathbb S$ is infinite.

In the second section of this paper, which is its main part, we do not assume that $A$ contains an $S$-stable basis;
we merely assume that $\mathbb S$ is not empty. We do, however, present an example in which the existence of an $S$-stable basis is assumed.
In the third section of this paper we assume that $A$ contains an $S$-stable basis


We present now some of the common definitions we use from order theory.
Let $P$ be a set. A relation $\leq$ on $P$ that is reflexive and transitive is called a quasi order or a preorder;
if the relation is also antisymmetric then it is called a partial order and $P$ is called a partially ordered set, or a poset.
Let $L \subseteq P$. We say that $a \in P$ is a lower (resp. upper) bound of $L$ if $a \leq x$ (resp. $a \geq x$) for all $x \in L$.
If the set of lower (resp. upper) bounds of $L$ has a unique greatest (resp. smallest) element, this element is called
the greatest lower (resp. least upper) bound of $L$, and is denoted by $\text{inf}L$ (resp. $\text{sup}L$).
We say that $L$ is a lower set if $L=\{ y \in P \mid y \leq x \text{ for some } x \in L\}$.
We say that $L$ is an upper set if $L=\{ y \in P \mid y \geq x \text{ for some } x \in L\}$.
A subset $\emptyset \neq D \subseteq P$ is called directed if for all $a,b \in D$ there exists $c \in D$ such that $a \leq c$ and $b \leq c$.
We say that $P$ is a dcpo (directed complete partial order) if
every directed subset of $P$ has a supremum.
A subset $I$ of $P$ is called an ideal of $P$ if $I$ is a lower set and directed.
$P$ is called an inf semilattice (resp. sup semilattice) if for all $a,b \in P$, $\text{inf}\{a,b\}$ (resp. $\text{sup}\{a,b\}$) exists in $P$.
If $P$ is both an inf semilattice and a sup semilattice, we say that $P$ is a lattice.
A subset $L$ of $P$ is called a sublattice of $P$ if $L$ is a lattice with respect to the partial order of $P$.
Similarly, one defines inf subsemilattice and sup subsemilattice.

We briefly discuss now the notion of an Alexandroff topological space.
A topological space whose set of open sets is closed under arbitrary intersections is called an Alexandroff space,
after P. alexandroff who first introduced such topological spaces
in his paper [Al] from 1937.
Equivalently, A topological space is called an Alexandroff space if every element has a minimal open set containing it.
A finite topological space is the most important particular case of an Alexandroff space.
In fact, Alexandroff spaces share many properties with finite topological spaces;
in particular, Alexandroff spaces have all the properties of finite spaces relevant for the theory of digital topology
(see [He] and [Kr]).
Thus, 
in the eighties the interest in Alexandroff spaces arose as a consequence of the
very important role of finite spaces in digital topology.
In 1999 F. G. Arenas studied the topological properties of Alexandroff spaces (see [Ar]).

Let $(T,\tau)$ be a topological space. For $X \subseteq T$ we denote by $clX$ the closure of $X$.
It is well known and easy to prove, that if one defines $x \leq_{\tau} y$
whenever $x \in cl \{ y \}$, then $ \leq_{\tau} $ is a quasi order; i.e., a reflexive and transitive relation.
$ \leq_{\tau}$  is called the specialization order.
Recall that $(T,\tau)$ is called $T_0$ if for every two distinct elements in $T$, there exists an open set containing one of them but not the other.
It is known that if $(T,\tau)$ is $T_0$ then $ \leq_{\tau} $ is a partial order.
On the other hand, for any quasi order $ \leq $ on a set $T$, one can define the topology whose open sets are the upper sets
of $T$ with respect to $ \leq $, denote it by $A_\leq$.
So, there are two functors, the specialization order from the class of all topological spaces to the class of all quasi ordered sets,
sending $(T,\tau)$ to $(T,\leq_\tau)$; and the functor in the opposite direction sending $(T,\leq)$ to $(T,A_\leq)$.
If one restrict the class of all topological spaces to the the class of all alexandroff topological spaces, then one has an isomorphism
between the categories.


In this paper the symbol $\subset$ means proper inclusion and the symbol $\subseteq$ means inclusion
or equality.

\section{The Alexandroff Topology}


Inspired by the Zarisky topology on the prime spectrum of a ring,
for every $M \subseteq A$ we denote by $V(M)$ the set of all $S$-nice subalgebras of $A$
containing $M$. It is easy to see that $ V(\{0\})=\mathbb S$, $V(F)=\emptyset$, and for every $M_1,M_2 \subseteq A$, we have $$
V(M_1 \cup M_2)=V(M_1) \cap V(M_2).$$
Thus, the set B=$\{ V(M) \}_{M \subseteq A}$ is a basis for 
a topology on $\mathbb S$. Namely, every open set
in $\mathbb S$ is a union of elements of $B$.
Moreover, for every set $\{M_i\}_{i \in I} $ of subsets of $ A$, we have $$V(\cup _{i \in I}M_i)= \cap_{i \in I} V(M_i).$$
Now, an intersection of union of elements of $B$ can be presented as a union of intersection of elements of $B$.
Indeed, let $I$ be a set and let $\{ J_i\}_{i \in I}$ be a set of sets such that for all $i \in I$ and $j_i \in J_i$ there exists a set  $ X_{i,j_i}$; then $$\bigcap_{i \in I} (\bigcup_{j_i \in J_i} X_{i,j_i})=\bigcup_{t \in \prod_{i \in I} J_i} (\bigcap_{i \in I} X_{i,t_i}),$$
where $t_i$ denotes the $i$-th component of $t$. We apply the above equation to elements of $B$ and deduce that every intersection of open sets of $\mathbb S$ is open. Thus, $\mathbb S$ is an Alexandroff topological space with respect to the topology defined above.


Let $T$ be a topological space and let $x,y \in T$; then $x \in cl\{y\} $ iff for every open set $U$ containing $x$, we have $y \in U$.
If $T$ is alexandroff, this is equivalent to $y \in U_x$, where $U_x$ denotes the minimal open set containing $x$.
Now, let $R \in \mathbb S$. It is easy to see that the minimal open set containing $R$ is $V(R)$. Indeed, let $U$ be an open set containing $R$; then $U$ is of the form $\cup_{i \in I} V(M_i)$ where $M_i \subseteq A$ for all $i \in I$. Thus, $R \in V(M_i)$ for some $i \in I$; hence $M_i \subseteq R$ and $V(R) \subseteq V(M_i)$.
On the other hand, $V(R) $ is an open set containing $R$.
Thus, the specialization order on $\mathbb S$ is the order of inclusion; i.e., for $R_1, R_2 \in \mathbb S$, $R_1 \leq R_2  $ iff $R_1 \subseteq R_2  $.
Moreover, as in any alexandroff topological space, $U \subseteq \mathbb S$  is open iff $U$ is an upper set with respect to the specialization order;
dually, $C \subseteq \mathbb S$  is closed iff $C$ is a lower set.





We will frequently use the following four basic lemmas.

\begin{lem} \label{basic lemma} 1. Let $R_1 $ and $R_2 $ be two elements of $ \mathbb S$  and let $R$ be an $S$-algebra satisfying $R_1 \subseteq R \subseteq R_2$;
then $R$ is an $S$-nice subalgebra of $A$. \linebreak
2. Let $\{R_i\}$ be a finite subset of $\mathbb S$; then $ \cap_{1 \leq i \leq n} R_i \in \mathbb S$.

\end{lem}

\begin{proof} Straightforward.

\end{proof}


For subsets $M \subseteq A$ and $T \subseteq F$ we define $$T \cdot M = \{\sum_{1 \leq i \leq n}t_i m_i \mid t_i \in T, m_i \in M\}.$$

\begin{lem} \label{union satisfies basic properties} Let $K=\{R_i\}_{i \in I}$ be a nonempty subset of $\mathbb S$ and denote $R_0=\cup_{i \in I} R_i$. Then
the following three properties are valid:

(a) $R_0 \cap F=S$;

(b) $F \cdot R_0=A$; and

(c) If $R_0$ is closed under addition then $S \cdot R_0 \subseteq R_0$.

In particular, if $R_0$ is a ring then it is an $S$-nice subalgebra of $A$.
\end{lem}

\begin{proof} (a) Clearly, $S \subseteq R_0 \cap F$. Let $\alpha \in R_0 \cap F$; then $\alpha \in R_1$ for some $R_1 \in K$. Since $R_1 \cap F=S$, we have $\alpha \in S$.
(b) Since $K$ is not empty there exists $R_1 \in K$ such that $FR_1=A$; thus, $A= FR_1 \subseteq FR_0 \subseteq A$.
(c) Let $\{r_j\}_{1 \leq j \leq n} \subseteq R_0$ and let $\{s_j\}_{1 \leq j \leq n} \subseteq S$; then for all $1 \leq j \leq n$, $r_j \in R_{i_{j}}$ for appropriate $R_{i_{j}} \in K$. Thus, for all $1 \leq j \leq n$, $s_j r_j \in R_{i_j} \subseteq R_0$; hence, $S \cdot R_0 \subseteq R_0$.
\end{proof}

Note that $R_0$ is not necessarily an $S$-algebra since it is not necessarily a ring.


\begin{lem} \label{sup of a chain exists} Let $C= \{  R_{i} \}_{i \in I}$ be a nonempty chain in $\mathbb S$. Then
the supremum of $C$ exists in $\mathbb S$.

\end{lem}

\begin{proof} Let $R_0=\cup_{i \in I}R_{i}$; since $C$ is a chain, $R_0$ is ring.
By Lemma \ref{union satisfies basic properties}, 
$R_0 \in \mathbb S$. It is clear that $R_0$ is the smallest $S$-nice subalgebra of $A$ containing every element of $C$.

\end{proof}

\begin{lem} \label{a maximal element exists in an open=upper set} Let $U$ be a nonempty open set of $\mathbb S$. Then
there exists a maximal element in $U$, which is also a maximal element in $\mathbb S$.

\end{lem}

\begin{proof} 
Consider $U$ with the partial order of containment. Let $C$ be a nonempty chain in $U$.
By Lemma \ref{sup of a chain exists}, $\text{sup}C$ exists in $\mathbb S$. Now, $\text{sup}C \supseteq R_1$ for some $R_1 \in U$;
thus, by the definition of an open set in $\mathbb S$,
$\text{sup}C \in U$. Therefore, by Zorn's Lemma, there exists a maximal element in $U$.

\end{proof}

\begin{rem} \label{in a maximal chain there exists a maximal element} In view of Lemma \ref{a maximal element exists in an open=upper set}, one can be more precise. In fact, for every maximal chain $C$ in a nonempty open set $U$ of $\mathbb S$ (there exists such a chain by Zorn's Lemma), $\text{sup}C=\cup_{R \in C}R$ is a maximal element of $C$, $U$ and $\mathbb S$.
Moreover, for the same reason, for any $R \in U$  there exists a maximal element $R' \in U$ containing $R$;
indeed, take any maximal chain in $U$ containing $R$, and use the reasoning above.

\end{rem}


Let $U \subseteq \mathbb S$ be an open set. We consider the following properties:

(a) $U$ is of the form $V(R_0)$ for some $R_0 \in \mathbb S$.

(b) $U$ is closed under arbitrary nonempty intersection.

(c) $V(\cap_{R \in U} R) = U$.

(d) $V(\cap_{R \in U} R)  \subseteq U$.

(e) $U$ is closed under finite nonempty intersection; i.e., $U$ is an inf subsemilattice of $\mathbb S$.

(f) $U$ contains no more than one minimal element.

\begin{prop} \label{a implies b implies c for open U} Notation as above; the following implications hold:

$(a) \Leftrightarrow (b) \Rightarrow (c) \Leftrightarrow (d) \Rightarrow (e) \Rightarrow (f)$.
\end{prop}

\begin{proof} $(a) \Rightarrow (b)$. By definition, every element of $U$ contains $R_0$.
Let $R'=\cap_{R \in W} R$ where $\emptyset \neq W \subseteq U$. Then $R_0 \subseteq R' \subseteq R''$ for some $R'' \in W$;
hence by Lemma \ref{basic lemma}, $R' \in \mathbb S$ and thus, by the definition of $U$, $R' \in U$.
$(b) \Rightarrow (a)$ and $(b) \Rightarrow (c)$. By assumption $\cap_{R \in U} R \in U \subseteq \mathbb S$, which is clearly the smallest element of $U$; since every $R \in U$ contains $\cap_{R \in U} R$ and $U$ is an open set, we get $U=V(\cap_{R \in U} R)$. 
$(c) \Rightarrow (d)$ is trivial.
$(d) \Rightarrow (c)$. As above, every $R \in U$ contains $\cap_{R \in U} R$; thus, $V(\cap_{R \in U} R) \supseteq U$.
$(d) \Rightarrow (e)$. Let $R_1,R_2 \in U$; by Lemma \ref{basic lemma}, $R_1 \cap R_2 \in \mathbb S$;  clearly, $R_1 \cap R_2 \supseteq \cap_{R \in U} R$,
and thus $R_1 \cap R_2 \in U$.
$(e) \Rightarrow (f)$. Assume to the contrary that there exists two different minimal elements $R_1 ,R_2$ in $U$.
Then, by assumption $R_1  \cap R_2 \in U$, but it is clearly strictly contained in both $R_1$ and $R_2$, a contradiction.

\end{proof}

We present now examples which demonstrate that the left to right implications in the previous proposition cannot be reversed.

\begin{ex} \label{left implications may not hold} To show that the implication $(c) \Rightarrow (b)$ may not hold, we can consider any case in which
$\mathbb S$ is taken as the open set and
$\cap_{R \in \mathbb S} R \notin \mathbb S$; it is clear that $V(\cap_{R \in \mathbb S} R)  = \mathbb S$.
As an explicit example, let $S= \mathbb{Z}$, $A=M_2(\mathbb{Q})$ and $U=\mathbb S$.
More generally, in [Sa3, discussion after Proposition 3.21] we showed that whenever
$A$ contains an $S$-stable basis, we have $\cap_{R \in \mathbb S} R \notin \mathbb S$.

To show that the implication $(e) \Rightarrow (d)$ may not hold, one can consider an infinite chain
of $S$-nice subalgebras of $A$ such that their intersection is an $S$-nice subalgebra of $A$ that is strictly contained in each of them.
As an explicit example, let $S= O_v$ be a valuation domain with value group $\mathbb{R}$, and let
$A=M_2(F)$, where $F$ is the field of fractions of $O_v$.
Take $0<r_0 \in \mathbb{R}$ and for all $0< r \leq r_0$ let $I_r=\{x \in O_v \mid v(x) \geq r\}$.
For all $0< r \leq r_0$ denote

\begin{equation*} R_r=
 \left(
\begin{array}{cccccccc}
O_v& I_r  \\
J_1 & O_v \\

\end{array} \right),
\end{equation*}
where $J_1$ is any nonzero ideal of $O_v$. Let $U=\cup_{0< r <r_0} V(R_r)$. Since $\{ I_r \}_{0<r<r_0}$ is a chain, $U$ is closed under finite nonempty intersection; however,
$$\cap_{R \in U} R=\cap_{0< r <r_0} R_r = R_{r_{0}} \notin U.$$

Finally we demonstrate that $(f)$ does not necessarily imply $(e)$. With the notation presented above,
let $J_2$ be an ideal of $O_v$ strictly containing $J_1$; and let

\begin{equation*} R'=
 \left(
\begin{array}{cccccccc}
O_v& I_{r_0}  \\
J_2 & O_v \\

\end{array} \right).
\end{equation*}
Let $U=\cup_{0< r <r_0} V(R_r) \cup V(R')$. Then the unique minimal element of $U$ is $R'$,
but $U$ in not closed under finite nonempty intersection since for all $0< r <r_0$, we have
 $R' \supset R' \cap R_r $; thus, $R' \cap R_r \notin U$

\end{ex}

\begin{prop} \label{exists a lower bound iff} Let $H$ be a nonempty subset of $\mathbb S$; then

1. There exists a lower bound for $H$ iff the infimum of $ H$ exists.

2. There exists an upper bound for $H$ iff the supremum of $ H$ exists.

\end{prop}

\begin{proof} Clearly, the right to left implication in both statements is trivial.
We prove the left to right implication of the first statement. Let $R_0 \in \mathbb S$  be a lower bound of $H$,
let $R_1 = \cap_{R \in H} R$ and let $R_2$ be any element of $H$ (note that $H$ is not empty).
Then by Lemma \ref{basic lemma}, 
$R_1 \in \mathbb S$  and it is clearly
the infimum of $H$.
We prove now the left to right implication of the second statement.
Let $K$ denote the set of all upper bounds of $H$; by assumption, $K$ is not empty.
Since $H$ is not empty, $K$ has a lower bound; indeed, any element of $H$ is a lower bound for $K$.
Thus, by the first statement, the infimum of $K$, $R_3 = \cap_{R \in K} R$, exists in $\mathbb S$.
Now, let $R \in H$, then for every $R' \in K$, we have $R \subseteq R'$; thus, $R \subseteq R_3$.
Therefore, $R_3 \in K$ and $R_3$ is the supremum of $H$.

\end{proof}

We note that the assumption that $H \neq \emptyset$ is crucial.
Indeed, the empty set clearly
has an upper bound but the supremum of $\emptyset$ does not exist, since $\mathbb S$ may not contain
a smallest element, as shown in Example \ref{left implications may not hold}.
Also, the empty set clearly
has a lower bound but the infimum of $\emptyset$ exists iff $\mathbb S$ is irreducible,
as we shall see in Theorem \ref{one maximal in U equivallent conditions}.

Dually to Proposition \ref{a implies b implies c for open U}, let $C \subseteq \mathbb S$ be a closed set. We consider the following properties:

(a) $C=cl \{R \}$ for some $R \in \mathbb S$.

(b) Every nonempty subset of $C$ has a supremum, which belongs to $C$.

(c) $\text{sup} C \in \mathbb{S}$ and $C=cl \{\text{sup} C\}$.

(d) $\text{sup} C \in \mathbb{S}$ and $cl \{\text{sup} C\} \subseteq C$.

(e) $C$ is a sup subsemilattice of $\mathbb S$.

(f) $C$ contains no more than one maximal element.

The proof of the following proposition is quite similar to the proof of Proposition \ref{a implies b implies c for open U}.
The implications $(a) \Rightarrow (b)$, $(c) \Rightarrow (b)$, and $(d) \Rightarrow (e)$ rely on
Proposition \ref{exists a lower bound iff}; we shall not prove it here.

\begin{prop} \label{a implies b implies c for closed C} Notation as above; the following implications hold:

$(a) \Leftrightarrow (b) \Leftrightarrow (c) \Leftrightarrow (d) \Rightarrow (e) \Rightarrow (f)$.
\end{prop}

Note the small difference 
in the implication $(c) \Rightarrow (b)$, which is valid in Proposition \ref{a implies b implies c for closed C} but not
in Proposition \ref{a implies b implies c for open U}. The reason for this difference is the fact that $V(M)$ is defined for all
$M \subseteq A$ whereas $cl \{ R\}$ is defined only for $R \in \mathbb{S}$. While in condition (c) of Proposition \ref{a implies b implies c for open U},
$ \cap_{R \in U} R$ is not necessarily in $\mathbb{S}$, in conditions (c) and (d) of Proposition \ref{a implies b implies c for closed C} we require that $\text{sup}C$ would be in $\mathbb{S}$.

We also note that condition (e) of Proposition \ref{a implies b implies c for closed C} implies that $C$ is
an ideal of $\mathbb{S}$, in the sense of order theory defined in the introduction.
Indeed, by assumption $C$ is closed and thus it is a lower set. By the assumption in (e),
$C$ is a sup subsemilattice of $\mathbb S$; in particular, $C$ is directed. In fact, in this case $C$ is actually a sublattice of $\mathbb S$.
To show this, let $R_1,R_2 \in C$, the infimum of $R_1$ and $R_2$ is their intersection
which is in $\mathbb{S}$; by assumption $C$ is a closed set and thus it is a lower set, so $\text{inf}\{R_1,R_2\} \in C$;
and by the assumption in condition (e), $C$ is a sup subsemilattice of $\mathbb S$.

The following theorem is important to our study.

\begin{thm} \label{irreducible has sup} Let $I $ be an irreducible subset of $ \mathbb S$.
Then $\cup_{R \in I}R$ is an $S$-nice subalgebra of $A$; in particular, $\text{sup}I=\cup_{R \in D}R$.

\end{thm}

\begin{proof} Denote $R_0=\cup_{R \in I}R$. Let $a,b \in R_0$; we prove that there exists
$R_1 \in I$ such that $a,b \in R_1$.
Assume to the contrary that there exists no such $R_1$. 
Let $G_1 $ denote the set of all elements in $\mathbb S$ not containing $a$,
and let $G_2 $ denote the set of all elements in $\mathbb S$ not containing $b$.
It is clear that $G_1$ and $G_2$ are closed in $\mathbb S$. 
However, by our assumption $I \subseteq G_1 \cup G_2$, while $I \nsubseteq G_1$ and $I \nsubseteq G_2$, a contradiction.
Thus, $R_0$ is a ring. By Lemma \ref{union satisfies basic properties}, $R_0 \in \mathbb S$.

\end{proof}

\begin{thm} \label{one maximal in U equivallent conditions} Let $U \subseteq \mathbb S$ be a nonempty open set.
The following statements are equivalent:

(a) $U$ has a greatest element.

(b) There exists $R_0 \in \mathbb S$ such that $clU=cl\{R_0\}$.

(c) $U$ has a unique maximal element.

(d) $U$ is irreducible.

(e) $U$ is a sup subsemilattice of $\mathbb S$. 

(f) Every nonempty subset $H \subseteq U$ has a supremum.

\end{thm}

\begin{proof} We prove $(a) \Rightarrow (b) \Rightarrow (c) \Rightarrow (d) \Rightarrow (e) \Rightarrow (f) \Rightarrow (a)$.
To show that $(a) \Rightarrow (b)$ we denote by $R_0$ the greatest element of $U$.
It is clear that every closed set containing $U$ also contains $R_0$; on the other hand,
since $R_0$ is the greatest element of $U$, by the definition of the topology on $\mathbb S$, every closed set containing  $R_0$
also contains $U$.
$(b) \Rightarrow (c)$. We prove that $R_0$ is the unique maximal element of $U$.
It is clear that $R_0$ is a maximal element of $U$, since otherwise $U \setminus cl\{R_0\} \neq \emptyset$ and thus $clU \setminus cl\{R_0\} \neq \emptyset$. 
Similarly, Assuming there exists another maximal element $R_1 \in U$, we get $R_1 \in U \setminus cl\{R_0\}$. 
$(c) \Rightarrow (d)$. Let $R_0$ denote the unique maximal element of $U$.
Assume to the contrary that
$U \subseteq G_1 \cup G_2$ where $G_1 , G_2$ are closed in $\mathbb S$
while $U \nsubseteq G_1$ and $U \nsubseteq G_2$. Let $R_1 \in U \setminus G_1$ and
$R_2 \in U \setminus G_2$. By Remark \ref{in a maximal chain there exists a maximal element},
there exist maximal elements $R_1 ', R_2 ' \in U$  containing $R_1$ and $R_2$, respectively.
Since $R_0$ is the unique maximal element of $U$, we have $R_1 '= R_2 '=R_0$.
Thus, $R_0 \notin G_1$ and $R_0 \notin G_2$, a contradiction. 
We prove $(d) \Rightarrow (e)$. Let
$R_1, R_2 \in U$; by assumption $U$ is irreducible and thus by Theorem \ref{irreducible has sup}, $R_0=\cup_{R \in U} R$ is
an $S$-nice subalgebra of $A$.
Clearly, $R_0 $ contains both $R_1$ and $R_2$. Thus, by Proposition \ref{exists a lower bound iff}, $\text{sup}\{R_1 ,R_2 \}$ exists.
Hence, by the definition of an open set, $\text{sup}\{R_1 ,R_2 \} \in U$.
We prove now $(e) \Rightarrow (f)$. 
Let $\emptyset \neq H \subseteq U$. By Remark \ref{in a maximal chain there exists a maximal element},
for every $R \in H$ there exists a maximal element $T_R \in U$ containing $R$.
By assumption, every two elements of $U$ have a supremum, thus these $T_R$ must all be equal.
So, $H$ is bounded from above and therefore, by Proposition \ref{exists a lower bound iff},
the supremum of $H$ exists.
Finally, we show $(f) \Rightarrow (a)$. By assumption $U$ has a supremum, and since $U$ is an open set, its supremum is its greatest
element. 

\end{proof}

In view of Theorem \ref{one maximal in U equivallent conditions}, we 
characterize now the irreducible components of $\mathbb S$.

\begin{prop} \label{irreducible components}  $I \subseteq \mathbb S$ is an irreducible component of $\mathbb S$
iff $I=cl \{ R \}$ for some maximal $R \in \mathbb S$.

\end{prop}

\begin{proof} Let $R$ be a maximal element of $\mathbb S$. It is clear that $cl \{ R \}$ is irreducible. 
Assume to the contrary that there exists an irreducible set $G \supset cl \{ R \}$. Let $R' \in G \setminus cl \{ R \}$.
Then $R \nsubseteq R'$ and $R' \nsubseteq R$. By Theorem \ref{irreducible has sup}, $\cup_{T \in G}T \in \mathbb S$.
However, $\cup_{T \in G}T$ strictly contains $R$, a contradiction.
On the other hand, let $I \subseteq \mathbb S$ be an irreducible component of $\mathbb S$.
Let $R_0=\cup_{T \in I}T$; by Theorem \ref{irreducible has sup}, $R_0\in \mathbb S$. Thus, $I \subseteq cl \{ R_0\}$. Since $cl \{ R_0\}$ is irreducible
and $I $ is an irreducible component of $\mathbb S$, one has $I=cl \{ R_0\}$.
Now, $R_0$ is maximal in $\mathbb S$, since otherwise there exists $R_0 \subset R_1 \in \mathbb S$, but then $cl \{ R_0 \} \subset cl \{ R_1 \}$,
a contradiction.

\end{proof}

Recall from [GHKLMS, Def. O-5.6.] that a topological space $T$ is called sober if for every irreducible closed subset $C$ of $T$,
there exists a unique $t \in T$ such that $C$
is the closure of $t$; i.e., $C$ has a unique generic point.
Also recall that a poset $P$ is called dcpo (directed complete partial order) if
every directed subset of $P$ has a supremum.
It is known (see, for example [GHKLMS, Ex. O-5.15.]) that every sober space is a dcpo, under the specialization order.

Note that by Lemma \ref{a maximal element exists in an open=upper set} there exists a maximal $S$-nice
subalgebra of $A$. In [Sa3, discussion after Corollary 3.21] we noted that in case $S$ is a valuation domain of $F$ and $A$ is a field, then the maximal $S$-nice subalgebras of $A$ are precisely the valuation domains (whose valuations
extend $v$) of $A$. So, by Proposition \ref{irreducible components}, the closures of these valuation domains
are precisely the irreducible components of $\mathbb S$.

 We also showed in [Sa3, Example 3.26]  that even in the case of a central
simple $F$-algebra, one can have an infinite ascending chain of
$S$-nice subalgebras of $A$ (even when $S$ is a valuation domain).
For the reader's convenience we present here the example.

\begin{ex} (cf. [Sa3, Example 3.26]) \label{example of an infinite accending chain of s-nice} Let $C$ be a non-Noetherian integral domain with field of fractions
$F$. Let $\{ 0 \} \neq I_{1} \subset I_{2} \subset I_{3} \subset
...$ be an infinite ascending chain of ideals of $C$ and
let $A=M_{n}(F)$. Then

\begin{equation*} R_1=
 \left(
\begin{array}{cccccccc}
C& C & ... & C & I_{1}  \\
C & C & ... & C & I_{1}  \\
. & . & ... & . & .  \\
. & . & ... & . & .  \\
. & . & ... & . & .  \\
C & C & ... & C & I_{1}  \\
C & C & ... & C & C  \\
\end{array} \right) \subset R_2=  \left(
\begin{array}{cccccccc}
C& C & ... & C & I_{2}  \\
C & C & ... & C & I_{2}  \\
. & . & ... & . & .  \\
. & . & ... & . & .  \\
. & . & ... & . & .  \\
C & C & ... & C & I_{2}  \\
C & C & ... & C & C  \\
\end{array} \right) \subset
...
\end{equation*}

 is an infinite accending chain of $C$-nice subalgebras
of $A$. So, let $$I=\cup_{n \in \mathbb N} cl \{ R_n\};$$ then $I$ is closed and irreducible with no generic point.

In particular, we have an example in which $\mathbb S$ is not sober. Nevertheless, in a subsequent paper
we will
show that $\mathbb S$ is indeed a dcpo and has some interesting properties from the point of view of domain theory.

\begin{rem} Note that by Lemma \ref{basic lemma}, $\mathbb S$ is an inf semilattice, where the infimum is actually an intersection of sets.
In particular, taking $U=\mathbb S$ in Theorem \ref{one maximal in U equivallent conditions},
the conditions of the theorem are also equivalent to the condition that $\mathbb S$ is a lattice.
In fact we can say even more. By the definition of an irreducible space, every finite intersection of nonempty open sets
is nonempty. In our case, whenever $\mathbb S$ satisfies the equivalent conditions of Theorem \ref{one maximal in U equivallent conditions},
the intersection of \textbf{all} nonempty open sets of $\mathbb S$ is the singleton $\{ R_0 \}$, where  $R_0$
denotes the greatest element of $\mathbb S$. Moreover, this property is also equivalent to the 
equivalent conditions presented in Theorem \ref{one maximal in U equivallent conditions}.

\end{rem}

Recall (cf. [En, Theorem 16.4]) that a valuation $v$ on a field $K$ is called henselian if $v$ extends uniquely to every algebraic field extension of $K$;
in this case, one also says that the corresponding valuation domain is henselian.
Thus, in view of the previous remark and the discussion before Example \ref{example of an infinite accending chain of s-nice},
we have,

\begin{ex}
If $S$ is an henselian valuation domain, $F$ is its field of fractions and $A$ is an algebraic field extension of $F$,
then $\mathbb S$ is a lattice (viewed from the point of view of order theory), and an irreducible topological space (viewed as an Alexandroff topological space). 

\end{ex}

\end{ex}

\section{Prime ideals}

In this short section we discuss 
the prime spectra of $S$-nice subalgebras of $A$ and the subsets of $\text{Spec}(S)$ covered by them.
We assume that $A$ is commutative (in this case whenever $R_1 \subset R_2$ are $S$-subalgebras of $A$
and $Q$ is a prime ideal of $R_2$, then $Q  \cap R_1$ is a prime ideal of $R_1$). 
We also assume that $A$ contains an $S$-stable basis, in order to be able to use the going down lemma for $S$-nice subalgebras of $A$ (Lemma \ref{going down}) and
Theorem \ref{every chain is covered commutative case}.

For a ring $T$, we denote by $\text{Spec}(T)$ the prime spectrum of $T$; i.e., the set of all prime ideals of $T$.
Recall the following definition from [Sa3]:
if for every $P \in \text{Spec}(S)$ there exists $Q \in \text{Spec}(R)$ lying over $P$, we say that $R$ satisfies ``Lying Over" (LO, in short) over $S$.
We denote by $\text{Spec}_R (S)$ the set of all prime ideals of $S$ having a prime ideal of $R$ lying over them;
namely,

$\text{Spec}_R (S)= \{ P \in \text{Spec}(S)  \mid$ there exists $Q \in \text{Spec}(R)$ lying over  $ P\}$.
Note that, by definition $\text{Spec}_R (S) \subset \text{Spec} (S)$ iff $R$ does not satisfy LO over $S$.


As usual, we use the term ``almost all" to mean that a property is satisfied to all but finitely many elements.

In the following lemma we show that whenever $R$ does not satisfy LO over $S$, there exists an $S$-nice subalgebra of $A$
whose prime spectrum lies over a larger set of primes of $S$ than the prime spectrum of $R$.

\begin{lem} \label{not LO then there exists a better one}  Let $R_1 $ be an $S$-nice subalgebra of $A$ such that
$\text{Spec}_{R_1} (S) \neq \text{Spec} (S)$,
and let $P \in \text{Spec} (S) \setminus \text{Spec} _{R_1} (S)$. Then there exists $R_1 \supset R \in \mathbb S$
such that $\text{Spec}_{R_1} (S) \subset \text{Spec}_{R} (S)$ and $P \in \text{Spec}_{R} (S)$.

\end{lem}

\begin{proof} Consider the chain $\{ P\}$. By Theorem \ref{every chain is covered commutative case}, there exists $R_2 \in \mathbb S$ having a prime ideal
lying over $P$. By Lemma \ref{basic lemma}, $R=R_1 \cap R_2$ is an $S$-nice subalgebra of $A$. Now, since $A$ is commutative,
every prime ideal of $R_i$ ($i=1,2$) intersect to a prime ideal of $R$. Thus, $\text{Spec}_{R_1} (S) \subset \text{Spec}_{R} (S)$ with
$P \in \text{Spec}_{R} (S)$.

\end{proof}

\begin{prop} \label{LO iff for almost all}  There exists $R \in \mathbb S$ satisfying LO over $S$ iff there exists $R_1 \in \mathbb S$
whose prime spectrum is lying over almost all prime ideals of $S$; namely, $\text{Spec}(S) \setminus \text{Spec}_{R_1} (S)$ is finite.

\end{prop}

\begin{proof} $(\Rightarrow)$ is trivial; just take $R_1=R$.
$(\Leftarrow)$ Let $R_1$ be an $S$-nice subalgebra of $A$ whose prime spectrum is lying over almost all prime ideals of $S$.
If $R_1$ satisfies LO over $S$ then we are done. Otherwise,
by Lemma \ref{not LO then there exists a better one} there exists $R_1 \supset R_2 \in \mathbb S$
with $\text{Spec}_{R_1} (S) \subset  \text{Spec}_{R_2} (S)$. After finitely many such steps we get $R_n \in \mathbb S$
that satisfies LO over $S$.

\end{proof}

\begin{rem}
Note that if there exists $R \in \mathbb S$ that  satisfies LO over $S$, then for every nonempty closed set $C \subseteq \mathbb S$ there exists $R_1 \in C$ that satisfies LO over $S$; indeed, take any $R_2 \in C$ and denote $R_1 = R \cap R_2$.
\end{rem}



In view of the previous remark, in the following proposition we present equivalent conditions for the non-existence of an $S$-nice subalgebra of $A$
that satisfies LO over $S$.

\begin{prop} The following conditions are equivalent:

(a) There exists no $R \in \mathbb S$ that  satisfies LO over $S$.

(b) There is no maximal subset $Y \subseteq \text{Spec}(S)$ (with respect to inclusion) such that
there exists $R \in \mathbb S$ with  $\text{Spec}_{R} (S)=Y$.

(c) For every $R \in \mathbb S$ there exists an infinite descending chain $\{R_i \}_{i \in I}$ of $S$-nice subalgebras of $A$
such that $R_i \subseteq R$ for all $i \in I$, and $\text{Spec}_{R_i} (S) \subset  \text{Spec}_{R_j} (S)$, whenever $R_j \subset R_i$.

(d) Every nonempty closed subset of $\mathbb S$ contains
an infinite descending chain $\{R_i \}_{i \in I}$ of $S$-nice subalgebras of $A$
such that $\text{Spec}_{R_i} (S) \subset  \text{Spec}_{R_j} (S)$, whenever $R_j \subset R_i$.

\end{prop}

\begin{proof} We prove $(a) \Rightarrow (b) \Rightarrow (c) \Rightarrow (d) \Rightarrow (a)$.
To show $(a) \Rightarrow (b)$, assume to the contrary that there exists $R_1 \in \mathbb S$ such that $Y=\text{Spec}_{R_1} (S)$
is a maximal subset of $\text{Spec}(S)$. By assumption $\text{Spec}_{R_1} (S) \neq \text{Spec} (S)$.
By Lemma \ref{not LO then there exists a better one} there exists $R_2 \in \mathbb S$ such that
$\text{Spec}_{R_1} (S) \subset \text{Spec}_{R_2} (S)$, a contradiction to the maximality of $Y$.
We prove now $(b) \Rightarrow (c)$. Let $R_1 \in \mathbb S$.
By assumption $\text{Spec}_{R_1} (S) \neq \text{Spec} (S)$, since otherwise $\text{Spec}_{R_1} (S)$ is a maximal subset of
$\text{Spec} (S)$. By Lemma \ref{not LO then there exists a better one} there exists $R_1 \supset R_2 \in \mathbb S$
such that $\text{Spec}_{R_1} (S) \subset \text{Spec}_{R_2} (S)$.
By assumption $\text{Spec}_{R_2} (S) \neq \text{Spec} (S)$; so
again by Lemma \ref{not LO then there exists a better one}, there exists
$R_2 \supset R_3 \in \mathbb S$
such that $\text{Spec}_{R_2} (S) \subset \text{Spec}_{R_3} (S)$. We continue this way to get an infinite
descending chain $\{R_k \}_{k \in \mathbb N}$ of $S$-nice subalgebras of $A$
such that $\text{Spec}_{R_i} (S) \subset \text{Spec}_{R_j} (S)$, whenever $i<j$.
To prove that $(c) \Rightarrow (d)$, let $C$ be a nonempty closed subset of $\mathbb S$ and let $R_1 \in C$.
The result follows by applying the assumption on $R_1$ and recalling the definition of the topology on $\mathbb S$.
Finally, we show $(d) \Rightarrow (a)$. Assume to the contrary that there exists $R_1 \in \mathbb S$ satisfying LO over $S$.
Then the closed subset $cl \{R_1 \}$ contains only elements that satisfy LO over $S$, a contradiction.




\end{proof}




\end{document}